\documentclass{amsart}
\usepackage{amsfonts}
\usepackage{graphicx}
\usepackage{amscd}
\usepackage{mathtools}

\newtheorem{theorem}{Theorem}
\theoremstyle{plain}

\newtheorem{definition}{Definition}
\newtheorem{example}{Example}

\newtheorem{lemma}{Lemma}

\newtheorem{proposition}{Proposition}
\newtheorem{remark}{Remark}

\numberwithin{equation}{section}

\newcommand{\abs}[1]{\lvert#1\rvert}

\begin{document}

\title[Ulam's method    for piecewise convex maps ]{ Ulam's method for computing stationary densities  of invariant measures   for piecewise 
convex maps with  countably infinite number of branches}

\author{Md Shafiqul Islam }

\address[Md Shafiqul Islam]{School of Mathematical and Computational Sciences\\ 
University of Prince Edward Island\\ 550 University Ave, Charlottetown, PE, C1A 4P3, Canada}
\email[Md Shafiqul Islam]{sislam@upei.ca}

\author[P. G\'ora]{Pawe\l\ G\'ora }
\address[P. G\'ora]{Department of Mathematics and Statistics, Concordia
University, 1455 de Maisonneuve Blvd. West, Montreal, Quebec H3G 1M8, Canada}
\email[P. G\'ora]{pawel.gora@concordia.ca}

\author{A H M Mahbubur Rahman}
\address[A H M Mahbubur Rahman]{Department of mathematics  and Statistics, \\Concordia University, 1455 De Maisonneuve Blvd. W.\\
Montreal, Quebec, H3G 1M8, Canada}
\email[A H M Mahbubur Rahman]{mahbubur.rahman@concordia.ca}


\begin{abstract}
Let  $\tau: I=[0, 1]\to [0, 1]$  be a piecewise convex map with countably infinite number of branches. In \cite{GIR},   the existence of absolutely continuous invariant measure (ACIM) $\mu$  for $\tau$ and the exactness of the system $(\tau, \mu)$  has been proven. In this paper, we develop an Ulam method for approximation  of $f^*$, the density of ACIM $\mu$.
 We   construct a sequence $\{\tau_n\}_{n=1}^\infty$ of maps  $\tau_n: I\to I$  s. t. $\tau_n$ has a finite number of branches  and the sequence  $\tau_n$  converges to $\tau$  almost uniformly.  Using supremum norms and Lasota-Yorke type inequalities,  we  prove  the existence  of ACIMs $\mu_n$ for $\tau_n$ with the   densities $f_n$.  For a fixed $n$, we apply Ulam's method with
$k$ subintervals to $\tau_n$  and compute  approximations $f_{n,k}$ of $f_n$. We prove that  $f_{n,k}\to f^*$ as $n\to \infty, k\to \infty,$
both a.e. and in $L^1$.\\
We provide examples of piecewise convex maps $\tau$ with countably infinite number of branches, their approximations $\tau_n$ with
finite number of branches  and for increasing  values of parameter $k$ show the errors $\|f^*-f_{n,k}\|_1$.
\end{abstract}

\maketitle

\section{Introduction} \label{Sec1}

 The existence and properties of  absolutely continuous invariant measures (ACIMs)   of  deterministic dynamical systems  reflect their long time behaviour and play an important role in understanding their chaotic nature \cite{BG1,LM,  LY}. Let  $\mathcal{B}$ be  the Lebesgue $\sigma$-algebra of subsets of $I=[0,1]$ and let $\lambda$ be the normalized Lebesgue measure on $I$. Let  $\tau : I \rightarrow I$ be a non-singular measurable transformation. A measure $\mu$ on $\mathcal{B}$ is $\tau$-invariant   if $\mu(\tau^{-1}(A))= \mu(A)$ for
all $A\in \mathcal{B}.$  The Frobenius-Perron (F-P) operator  $P_\tau: L^1(I, \mathcal{B}, \lambda)\to L^1(I, \mathcal{B}, \lambda) $ of  $\tau$   plays an important role for the existence, approximations and properties of ACIMs. The F-P operator  $P_\tau$ is defined by \begin{equation}\label{FP-Operator}
 \int_A P_{\tau} f \hspace{1.5 mm}d \lambda = \int_{\tau^{-1}(A)} f \hspace{1.5 mm} d\lambda,\hspace{1 mm}
\forall A \in \mathcal{B}, \forall f \in L^{1}.
\end{equation}
It can be shown that for $\tau$ with countable number of monotonic branches, the F-P operator $P_\tau$ has the following representation \cite{BG1} :
\begin{equation}P_{\tau}f(x)=\sum _{i\ge 1}\frac{f(\tau_{i}^{-1}(x))}
{|{\tau^\prime(\tau_{i}^{-1}(x))}|} \chi_{I_i}(x)=\sum_{z\in\{\tau^{-1}(x)\}}\frac{f(z)}{\abs{\tau^{\prime}(z)}},
\end{equation}
where $\tau_{i}^{-1}$ are inverse branches of $\tau_i=\tau_{| I_i}$ , $i\ge 1$, and $\{I_i\}_{i\ge 1} $ are the subintervals of the monotonicity partition.
The F-P operator  $P_\tau$ has a number of useful properties including: (i)  {Linearity:}   $P_\tau:L^1 \rightarrow L^1 $ is a linear operator; (ii) {Positivity:} If $f\in L^1$ and $f\ge 0$, then $P_\tau f\ge 0$; (iii) { Contractivity:}   $P_\tau:L^1\rightarrow L^1 $ is a contraction, i.e., $\parallel{P_\tau f}\parallel_1\le \parallel{ f}\parallel_1$ for any $f\in L^1$. Moreover,                        $P_\tau:L^1\rightarrow L^1 $  is continuous with respect to norm topology;
(iv)  $P_\tau$ preserves integrals, i.e., $\int_0^1P_\tau f d\lambda=\int_0^1 f d\lambda$;
(v)  If $\tau_1,\tau_2:I\rightarrow I$ are non-singular transformations, then                       $P_{\tau_1\circ\tau_2}f=P_{\tau_1}\circ P_{\tau_2}f$. In particular, $P_{\tau^n}f=P^n_{\tau}f$ ;
(vi) Operator $P_\tau$ is conjugated to the Koopman operator $Kg=g\circ \tau$ on $L^\infty$, i.e., $$\int_I g\cdot P_\tau f\, d\lambda= \int_I g\circ\tau \cdot f\, d\lambda\ , f\in L^1\ , \ g\in L^\infty ;$$
(vii)      $P_\tau$ has a fixed point $f^*$ if and only if  the measure
$\mu=f^*\cdot\lambda$ defined by $\mu(A)=\int_{A}f^*d\lambda$ is
$\tau-$invariant. From the   property (vii) above  one can see  that the fixed point  of 
 $P_{\tau }$ is  the  density $f^*$ of  the ACIM  $\mu$ of $\tau$ (see 
 \cite{BG1}).

  In \cite{LY}, Lasota and Yorke proved the existence of ACIMs for piecewise expanding maps. In this paper, we consider transformations that are not necessarily expanding, i.e., their derivatives may be smaller than $1$, 
 but they possess another property which makes them very special, namely piecewise convexity.
 In \cite {LY1},  the authors  studied the existence of  ACIMs and the exactness for piecewise convex transformations with a finite number of branches and with a strong repeller. The authors in \cite{LY1} considered the following properties as  key properties for the proof of  the existence of  ACIMs  (i)  the   F-P operator $P_\tau$ maps non-increasing functions to non-increasing functions; (ii) If $f: [0, 1]\to \mathbb{R}^+$  is a non-increasing function, then 
$$\parallel P_\tau f\parallel_\infty \le A\parallel f\parallel_\infty+B\parallel f \parallel_1,$$
 where $A<1$ and $B>0$ are some constants.  In \cite{HR}, the author studied   ACIMs for piecewise convex maps with infinite number of branches, where $1$ is the limit point of partition points.  
The literature on  piecewise convex maps with a  finite number of branches is quite rich, see for example \cite{BDSS}, \cite{D},
\cite{In},  \cite{In2}, \cite{JM}, \cite{JM1}.

 In \cite{GIR},  we have studied ACIMs of maps in   two classes  $\mathcal{T}_{pc}^\infty(I), \mathcal{T}_{pc}^{\infty, 0}(I) $ of piecewise convex maps   $\tau: I \to I$ with countably infinite number of branches.  Any $\tau$ in the first class $\mathcal{T}_{pc}^\infty(I)$ is a   piecewise convex map   $\tau: I \to I$ with a countable number of branches and   arbitrary countable number of limit points of partition points  separated from $0$.  Let $\{ 0=a_0, a_1, a_2, \dots, a_n, \dots \}$ be a countable partition of $I$ such that $a_0<a_1$ and all $a_2, a_3, \dots \in [a_1, 1].$  We do not assume the sequence 
 $\{ a_2, \dots, a_n, \dots \}$ to be increasing or decreasing. For any $i\in \{0, 1, 2, \dots\},$ let $n(i)$ be the index such that the interval  $[a_i, a_{n(i)}]$ does not contain any other points of the partition. If $a_k$ is the limit point of decreasing subsequence of $a_{n}$'s, the $n(k)$ is not defined. For such points 
the interval $[a_k,a_k]$ is not considered an  element of the partition.  We allow any countable number of the limit points of the points of the partition. We assume that the union of intervals $[a_i, a_{n(i)}]$ covers the whole interval $I$ except of a countable number of points.
At any such point $x$ we assume $\tau(x)=0$ and $\tau'(x)=+\infty$. These points do not contribute to the images of the $P_\tau$
operator.

 Any $\tau$ in  the second  class $ \mathcal{T}_{pc}^{\infty, 0}(I)$  is  a   piecewise convex map   $\tau: I \to I$ with  countably infinite number of branches  such that there exists a  countable partition  $\{ 0=a_0< \dots < a_{0, -n}< a_{0, -(n-1)}<\dots <a_{0, -2}< a_{0, -1}=a_1, a_2, a_3, \dots, a_n, \dots\}$    of $I$ with  $\lim_{n\to\infty}a_{0, -n}=0.$  It is proved in \cite{GIR}  that any $\tau\in \mathcal{T}_{pc}^\infty(I)\cup \mathcal{T}_{pc}^{\infty, 0}(I) $ has the unique  stationary density $f^*$  of the ACIM $\mu^*.$   In this paper, we apply  Ulam's method for approximation  of $f^*.$  Fixed points of the F-P operator  $P_\tau$ of $\tau\in \mathcal{T}_{pc}^\infty(I)\cup \mathcal{T}_{pc}^{\infty, 0}(I) $ are the stationary densities of $\tau.$  The F-P operator  $P_\tau$  is an infinite dimensional operator. Except some simple cases (where $\tau$ is piecewise linear Markov and the Frobenius-Perron operator has a matrix representation), it is not easy to obtain an analytical solution of the  F-P equation $P_\tau f=f.$  
  
 \bigskip A method for  numerical computations  of invariant measures of dynamical systems was
suggested by Ulam  \cite{U}. For   piecewise expanding deterministic  transformations \cite{LY}, T-Y Li \cite{TYL} first
proved the convergence of Ulam's approximation. Since then, Ulam's method has been applied to one and higher dimensional expanding deterministic transformations (see,  for example, \cite{BL, BM, BS, DZ}.  Ulam's method is one of the most used and the best understood methods among numerical methods for the approximation of ACIMs for  deterministic maps \cite{TYL, U}.  However, approximate invariant densities via Ulam's methods are piecewise constant on subintervals of the partition of the state space.  In \cite{M}, Miller proved the convergence of Ulam's method for piecewise convex transformations with a finite number of branches and a strong repeller.  J. Ding  \cite{D} developed and presented piecewise linear and piecewise quadratic Markov finite approximation methods for such piecewise convex transformations. 

 In Section \ref{Sec2}, we introduce notations and review the existence of  stationary densities   of ACIMs for  $\tau\in \mathcal{T}_{pc}^\infty(I)\cup \mathcal{T}_{pc}^{\infty, 0}(I).$ In Section \ref{Sec3}, we   construct a sequence $\{\tau_n\}_{n=1}^\infty$ of maps  $\tau_n: I \to I$  s. t. $\tau_n$ has  a finite number of branches  and the sequence $\tau_n$  converges to $\tau$  almost uniformly.  Using supremum norms and Lasota-Yorke type inequalities,  we  prove  the existence of  densities $f_n$ of ACIMs $\mu_n$ for $\tau_n$. Moreover, we  prove that $f_n$ converges to $f^*.$  In Section \ref{Sec5} we apply Ulam's method with $k$ subintervals to $\tau_n$, compute an approximation $f_{n,k}$ of $f_n$ and prove that  $f_{n,k}\to f_n$ as $ k\to \infty.$ 
In Theorem \ref{Thm5} we prove that  $f_{n,k}\to f^*$, as $ n,k\to \infty$,  both a.e. and   in $L^1$.
Theoretically, we could apply Ulam's method directly to the map $\tau$. Practically, such method would be impossible since the calculation of the Ulam's matrix $M_\tau^{(k)}$ would require solving of infinitely many equations.
In Section \ref{Sec6}, we present numerical examples.

   \section { ACIMs for  piecewise convex maps  in classes $\mathcal{T}_{pc}^\infty(I)$ and $\mathcal{T}_{pc}^{\infty, 0}(I)$}\label{Sec2}
In this section, we review results on the existence of stationary densities of ACIMs  of piecewise convex maps with countably infinite number of branches. We closely follow \cite{GIR}.

  \subsection { ACIMs for  piecewise convex maps in   $\mathcal{T}_{pc}^\infty(I) $ }

For the maps in $\mathcal{T}_{pc}^\infty(I) $  the limit points of the partition points  are separated from $0$.

We say that $\tau\in \mathcal{T}_{pc}^\infty(I)$  if 
\begin{enumerate}
 \item     $\tau_{0}=\tau|_{[0, a_{1})}$ is continuous and convex;  \\  
 $\tau_{i}=\tau|_{[a_{i}, a_{n(i)})}$ is continuous and convex, $i=1, 2, \cdots;$  \\  
 
\item   $\tau(a_{i})=0, \tau^\prime(a_{i})>0,  i=1, 2, \dots;$\\
  
\item   $\tau (0)=0, \tau^\prime (0)=\alpha_{1}>1;$ \\

\item   $\sum_{i=1}^\infty \frac{1}{\tau^\prime (a_{i})}<\infty.$ \\
\end{enumerate}
For $\tau\in \mathcal{T}_{pc}^\infty(I), f\in L^1(I), f\ge 0$ the F-P operator $P_{\tau}$ can be represented as 
\begin{equation}
P_{\tau}f(x)=\frac{f(\tau_{0}^{-1}(x))}{\tau^\prime(\tau_{0}^{-1}(x))}\chi_{\tau[0, a_{1})} (x) 
+\sum_{i=1}^\infty\frac{f(\tau_{i}^{-1}(x))}{\tau^\prime(\tau_{i}^{-1}(x))}\chi_{\tau[a_{i}, a_{n(i)})} (x).
\end{equation}

The following results are proved in \cite{GIR}.

\begin{lemma}
Let $\tau\in \mathcal{T}_{pc}^\infty(I), f\in L^1(I), f\ge 0, f $  be  non-increasing. Then 
\begin{enumerate}
\item $P_{\tau}(f)\in L^1(I).$
\item $P_{\tau}(f)\ge 0.$
\item $P_{\tau}(f)$ is non-increasing.
\item $\parallel P_{\tau}(f)\parallel_\infty\le C\parallel f\parallel_\infty,$  where $C=\left(\frac{1}{\alpha_{ 1}}+\sum_{i=1}^\infty \frac{1}{\tau^\prime (a_{i})} \right).$
\end{enumerate}
\end{lemma}

\begin{proposition} 
If  $f\ge 0 $  and $f$ is non-increasing, then $f(x)\le \frac 1x \lambda(f), \text~{for}~ x\in [0, 1],$ where 
$$\lambda(f)=\int_0^1 f(x)d\lambda(x).$$ \end{proposition}

\begin{lemma}\label{Le4} If $f: I \to \mathbb{R}^+$ is non-increasing and 
$\tau\in \mathcal{T}_{pc}^\infty(I),$ then 
 \begin{equation}\label{LYI}\parallel P_{\tau}(f)\parallel_\infty\le \frac{1}{\alpha_{ 1}}\parallel f\parallel_\infty + D\parallel f\parallel_1, \end{equation}    where $D=\left(\sum_{i=1}^\infty \frac{1}{a_{i}}\frac{1}{\tau^\prime(a_{i})}\right).$
\end{lemma}

\begin{theorem}\label{ThmSec2}
Let $\tau\in \mathcal{T}_{pc}^\infty(I).$ Then $\tau$ admits an ACIM $\mu=f^*\cdot \lambda$ with non-increasing density function  $f^*.$ The ACIM $\mu$ is unique and the system $(\tau, \mu)$ is exact.
\end{theorem}

\subsection
{ACIMs for piecewise convex maps  in $\mathcal{T}_{pc}^{\infty, 0}(I)$}

For the maps  in $\mathcal{T}_{pc}^{\infty, 0}(I)$  the point $0$ is  the limit point of partition points.

 We say that $\tau\in \mathcal{T}_{pc}^{\infty,0}(I)$  if 
\begin{enumerate}
 \item     $\tau_{ -j}=\tau|_{[a_{0, -(j+1)}, a_{0,-j})}$ is continuous and convex, $ j=1, 2, \dots$ ;  \\  
 $\tau_{i}=\tau|_{[a_{i}, a_{n(i)})}$ is continuous and convex,  $i=1, 2, \cdots;$  \\  
 
\item   $\tau(a_{0,-j})=0, \tau^\prime(a_{0,-j})>0,  j=1, 2, \dots;$\\

$\tau(a_{i})=0, \tau^\prime(a_{i})>0,  i=1, 2, \dots;$\\

\item   $\sum_{i=1}^\infty \frac{1}{\tau^\prime(a_{i})} < \infty;$\\

\item $D_{1}=\sum_{j=1}^\infty \frac{1}{\tau^\prime(a_{0,-j})}< 1.$ \\
\end{enumerate}

\begin{remark} The Condition 3 and  the Condition 4 can be replaced by the following condition:
$$(3^+) \ \ \        \sum_{j=1}^\infty \frac{1}
{\tau'(a_{0,-j})}+    \sum_{i=1}^\infty \frac{1}{\tau^\prime (a_{i})}<\infty.                     $$\\
If $(3^+)$ is satisfied, then we can find  an integer  $J\ge 1$ such that
$$ \sum_{j=J}^\infty \frac{1}
{\tau'(a_{0,-j})}< 1 ,$$
and after proper renaming of  partition points, the Condition 3 and  the Condition 4 are satisfied.
\end{remark} 
The following results are proved in \cite{GIR}.
\begin{lemma}\label{Lem6} 
Let $\tau\in \mathcal{T}_{pc}^{\infty,0}(I), f\in L^1(I), f\ge 0, f $ be non-increasing. Then 
\begin{enumerate}
\item $P_{\tau}(f)\in L^1(I).$
\item $P_{\tau}(f)\ge 0.$
\item $P_{\tau}(f)$ is non-increasing.
\item $\parallel P_{\tau}(f)\parallel_\infty\le C\parallel f\parallel_\infty,$  where $C=\left(D_1+\sum_{i=1}^\infty \frac{1}{\tau^\prime (a_{i})} \right).$
\end{enumerate}
\end{lemma}

\begin{lemma}\label{Lem7} If  $f: [0, 1] \to \mathbb{R}^+$ is non-increasing and $\tau\in \mathcal{T}_{pc}^{\infty, 0}(I),$   then   \begin{equation}\label{LYI}\parallel P_{\tau}(f)\parallel_\infty\le D_{1}\parallel f\parallel_\infty +D\parallel f\parallel_1, \end{equation}    where $D=\sum_{i=1}^\infty \frac{1}{a_{i}\tau^\prime(a_{i})}.$

\end{lemma}


\begin{theorem}\label{ThmSec3}
Let $\tau\in \mathcal{T}_{pc}^{\infty,0}(I).$ Then $\tau$ admits an ACIM $\mu=f^*\cdot \lambda$ with non-increasing density function  $f^*.$  The ACIM $\mu$ is unique and the system $(\tau, \mu)$ is exact.
\end{theorem}


  \section{Ulam's method for  piecewise convex maps with countable number of branches }\label{Sec3}
  In this section our main goal is to develop Ulam's method for stationary densities of ACIMs for piecewise convex maps with countably infinite number of branches. We do this in the following two steps.
  \subsection{Approximation of  maps with infinite number of branches} \label{Sec4}

In this section we approximate piecewise convex maps with  countably infinite number of branches by piecewise convex maps with a finite number of branches.

  We concentrate on a case of a map $\tau \in  \mathcal T^{\infty,0}_{pc}(I)$.  On the example of point 0, we show how to perform
an approximation of $\tau$ with a point of a partition that is a limit point of other points of the partition. The method can be applied to maps in  $\mathcal T^{\infty}_{pc}(I)$,
that have points with this property without substantial changes.

Let
$\tau\in \mathcal T^{\infty,0}_{pc}(I)$.  We assume that there are no other limit points of the partition points.
  For simplicity we 
rename the partition points. The new partition points are $a_n$, $n=0,1,2,\dots$, with $a_0=1$, $a_{n+1}<a_n$, $n\ge 0$ and
$\lim_{n\to\infty} a_n=0$. 
Then, the assumptions (3) and (4) of Section 2.2 are restated as:\\
(3') There exists an $N\ge 1$ such that  
\begin{equation}\label {new constantsD1} D_1=\sum_{n=N+1}^\infty \frac 1{\tau'(a_n)} < 1 .
\end{equation}

 Then, the Lemma 3 and the Lemma 4 of Section 2.2 hold with changed constants
\begin{equation}\begin{split}\label {new constantsCD}
C&=D_1+ \sum_{n=1}^N \frac {1}{\tau'(a_n)}=\sum_{n=1}^\infty \frac {1}{\tau'(a_n)},\\
D&=\sum_{n=1}^N \frac 1{a_n\tau'(a_n)}.
\end{split}\end{equation}

For $n\ge N$, we   construct a sequence $\{\tau_n\}_{n=N}^\infty$ of maps  $\tau_n: I \to I $  s.t.   $\tau_n$ has a  finite number of branches  and the sequence  $\tau_n$  converges to $\tau$  almost uniformly.  Using supremum norms and Lasota-Yorke type inequalities,  we  prove the existence of  stationary densities $f_n$ of ACIMs $\mu_n$ for $\tau_n$. We  approximate $\tau: I \to I$ with the following sequence  of maps $\tau_n: I \to I$, $n\ge N$,  with a finite number of branches:  
\begin{equation}\label{seqapprox} \tau_n(x)=\begin {cases} x/a_n  \ ,         \  0\le x< a_n;\\
                           \tau(x) \ ,         \  a_n\le x\le 1.
\end{cases}
\end{equation}
In the following, we show that for each $n\ge 0,$ the map $\tau_n$  has an ACIM. 
It can be easily shown that each $\tau_n$  is a piecewise convex maps with  the finite partition $\{1=a_0, a_1, a_2, \cdots, a_n, a_{n+1}= 0\}$   and $\tau_n$ satisfies  following conditions:
\begin{enumerate}
 \item      $\tau_{n_j}=\tau_n|_{{(a_{j}, a_{j-1}]}}$ is continuous and convex, $j= 1, 2, \cdots, n+1;$  \\  
\item   $\tau_n(a_{j})=0, \tau_n^\prime(a_{j})>0,  j=1, \cdots, n+1;$\\
\item   $\tau_n (0)=0, \tau_n^\prime (0)=\frac{1}{a_n}>1.$ \\

\item   $\sum_{j=1}^{n+1} \frac{1}{\tau^\prime (a_{j})}<\infty.$ \\
\end{enumerate}
Let  $ f\in L^1(I), f\ge 0.$  Then, the F-P operator $P_{\tau_n}$ is defined as 
\begin{equation}\label{F-P tau_n}
P_{\tau_n}f(x)=\frac{f(\tau_{a}^{-1}(x))}{1/a_n}+\sum_{j=1}^n\frac{f(\tau_{n_j}^{-1}(x))}{\tau_n^\prime(\tau_{n_j}^{-1}(x))}\chi_{\tau_n(a_{{j}}, a_{{j-1}}]} (x),
\end{equation}
where $\tau_a= x/a_n $ on the interval $[0,a_n)$.
\begin{lemma}\label{Lem1}                  
Let $ f\in L^1(I), f\ge 0, f $ be a non-increasing function. Then 
\begin{enumerate}
\item $P_{\tau_n}(f)\in L^1(I) $ for each $n=1, 2, \dots.$
\item $P_{\tau_n}(f)\ge 0$  for each $ n=1, 2, \dots.$
\item $P_{\tau_n}(f)$ is non-increasing for each $ n=1, 2, \dots.$
\item $\parallel P_{\tau_n}(f)\parallel_\infty\le C_n \parallel f\parallel_\infty \le \left(1+C \right) \parallel f\parallel_\infty$ for each $ n=1, 2, \dots$ where $C_n=\left(a_n+\sum_{j=1}^{n} \frac{1}{\tau_n^\prime (a_{{j}})} \right).$
\end{enumerate}
\end{lemma}
\begin{proof}
\begin{enumerate}
\item Holds for the F-P operator of any map.
\item  
$P_{\tau_n}f(x)$ is shown in formula (\ref{F-P tau_n}) and each summand is clearly non-negative.

\item  It can be easily shown that each summand in formula  (\ref{F-P tau_n})    is non-increasing.
\item We have

\begin{eqnarray*}
P_{\tau_n}f(x)&=&\frac{f(\tau_{a}^{-1}(x))}{1/a_n}+\sum_{j=1}^n\frac{f(\tau_{n_j}^{-1}(x))}{\tau_n^\prime(\tau_{n_j}^{-1}(x))}\chi_{\tau_n(a_{{j}}, a_{{j-1}}]} (x)\\
&\le& a_n \parallel f\parallel_\infty+ \sum_{j=1}^{n}\frac{\parallel f\parallel_\infty}{\tau_n^\prime(\tau_{n_j}^{-1}(x))}\le  a_n \parallel f\parallel_\infty+\sum_{j=1}^{n}\frac{\parallel f\parallel_\infty}{\tau_n^\prime(a_{{j}})}\nonumber\\
&=& \left(a_n+\sum_{j=1}^{n} \frac{1}{\tau_n^\prime (a_{{j}})} \right)\parallel f\parallel_\infty
\le \left(1+C\right)\parallel f \parallel_{\infty}
\end{eqnarray*}
\end{enumerate}

\end{proof}

\begin{lemma}\label{Lem3} If  $f: [0, 1] \to \mathbb{R}^+$ is non-increasing,   then  for each $n\ge N$,
  \begin{equation}\label{unif est0}
  \parallel P_{\tau_n}(f)\parallel_\infty\le (a_n+D_1)\parallel f\parallel_\infty +D\parallel f\parallel_1, 
  \end{equation}

\end{lemma}

\begin{proof}
Since $f$ is non-increasing, $f(0)\ge \parallel f\parallel_\infty,$ and by Lemma \ref{Lem1},
$P_{\tau_n}f(0)= \parallel P_{\tau_n}f\parallel_\infty.$ 
Now, 
\begin{eqnarray*}
P_{\tau_n}f(0)&=&\frac{f(\tau_{a}^{-1}(x))}{1/a_n}+\sum_{j=1}^{n}\frac{f(\tau_{n_j}^{-1}(0))}{\tau_n^\prime(\tau_{n_j}^{-1}(0))}\chi_{\tau_n(a_{{j}}, a_{{j-1}}]} (0)\nonumber\\
&=&\frac{1}{\tau_n^\prime(0)}f(0)+\sum_{j=1}^N\frac{f(\tau_{n_j}^{-1}(0))}{\tau_n^\prime(\tau_{n_j}^{-1}(0))}
+\sum_{j=N+1}^n\frac{f(\tau_{n_j}^{-1}(0))}{\tau_n^\prime(\tau_{n_j}^{-1}(0))}\\
&\le&a_nf(0)+\sum_{j=1}^N\frac{f(a_{n_{j}})}{\tau_n^\prime(a_{n_{j}})}+
\sum_{j=N+1}^n\frac{f(a_{n_{j}})}{\tau_n^\prime(a_{n_{j}})}\nonumber\\
&\le& ( a_n+ D_1)f(0)+\sum_{j=1}^N \frac{\lambda(f)}{a_{{j}}}\frac{1}{\tau_n^\prime(a_{{j}})}\nonumber\\
&\le&( a_n+ D_1)\parallel f\parallel_\infty +D\parallel f\parallel_1.
\end{eqnarray*}
\end{proof}

\begin{theorem}\label{Thm3}
For each $n\in \mathbb{N},$ $\tau_n$ admits an ACIM $\mu_n=f_n^*\cdot \lambda$ with non-increasing density function  $f_n^*.$
\end{theorem}

\begin{proof}
Our proof works for all $n\ge n_0$, where we have $a_{n_0}+D_1<1$ and gives a uniform estimate. For the $n<n_0$ the claim follows from Subsection 2.1.

Let $f\in L^1$ be non-negative and non-increasing. Consider the sequence $\{P_{\tau_n}^k f\}_{k=0}^\infty.$ Since every $P_{\tau_n}^k f$ is non-negative and non-increasing, by part (3) of Lemma \ref{Lem1} we can apply Lemma \ref{Lem3} iteratively and obtain 
\begin{eqnarray*}
&&\parallel P_{\tau_n}^k f\parallel_\infty
\le(a_{n_0}+D_1)^k  \parallel  f\parallel_\infty+\frac{D}{1-(a_{n_0}+D_1)} \parallel  f\parallel_1.
\end{eqnarray*}
So  the sequence $\{P_{\tau_n}^kf\}_{k=0}^\infty$ is uniformly bounded and weakly compact. By Yosida-Kakutani Theorem \cite{YK}, $\frac 1k\sum_{j=1}^k P_{\tau_n}^jf$
converges in $L^1$ to a $P_{\tau_n}$ invariant function $f_n^*.$ It is non-increasing since it is the limit of non-increasing functions.
\end{proof}
\noindent 

\begin{definition}
Let $\tau$  and  $\tau_n$, $n=1,2,\dots$  be maps  from  $I$ into itself. We say that $\tau_n$ converges to $\tau$ almost uniformly if for any $\epsilon>0,$ there exists a measurable set $A_{\epsilon}\subset [0,1], \lambda(A_\epsilon)>1-\epsilon, $  such that $\tau_n\to \tau$ uniformly on $A_\epsilon.$
\end{definition}
\begin{lemma}\label{Lem4} 
Let $\tau$ be a piecewise convex map with countably many branches and the sequence $\tau_n$, $n=1,2,\dots$ be defined by formula (\ref {seqapprox}). Then,
$\tau_n$ converges to $\tau$ almost uniformly.
\end{lemma}

\begin{proof}
Let $\epsilon>0.$ Choose the  decreasing partition $\{1=a_0, a_1, a_2, \cdots, a_n, a_{n+1}= 0\}$ of $[0,1]$ for $\tau_n$ such that $a_n<\epsilon.$  Let $A_\epsilon=(a_n,1).$  Then, $\lambda(A_\epsilon)=1-a_n>1-\epsilon $  and we have  $\tau_n= \tau$  on $A_\epsilon.$\\
\end{proof}

\subsection{Ulam's method} \label{Sec5}   
In this section, first, we describe Ulam's method for a finite dimensional approximation $P_{n, k}$ of the Perron-Frobenius operator $P_{\tau_n}$ of $\tau_n.$ Ulam's method computes  $f_{n,k}$ on a partition  of $k$ subintervals of the state space   as  an approximation of the actual   stationary   density function  $f_n$  of  $\tau_n, ~n\ge 1.$  Moreover,  we show that $f_{n, k}$  converges to $f_n$ as $k\to \infty.$ We  follow \cite{TYL}, \cite{M} and \cite{D}. In general, the Ulam's method is widely studied and used, see for example
\cite{BM}, \cite{BS},   \cite{DZ}, \cite{F}, \cite{GB} and the references within.

Let $\tau_n$ be an approximation of $\tau\in \mathcal{T}_{pc}^\infty(I).$
 Then, by the Theorem \ref{Thm3},  $\tau_n$ has an ACIM $\mu_n$ with a stationary density function $f_n.$  Now, we  describe Ulam's method for approximating $f_n.$  Let $k$ be a positive integer. Let $\mathcal{P}^{(k)} = \{J_1, J_2,\dots ,J_k\}$ be a partition of the
interval $[0, 1]$ into $k$ equal subintervals. Now, construct  the matrix
$$\mathbb{M}^{(k)}_{n}=\left(\frac{\lambda\left(\tau_{n}^{-1}(J_j)\cap J_i\right)}{\lambda(J_i)}\right)_{1\le i,j\le k}.$$ Let $L^{(k)}\subset   L^1([0, 1],\lambda)$ be a subspace of $L^1$ consisting of functions which are constant on
elements of the partition $\mathcal{P}^{(k)}.$ We will represent functions in $L^{(k)}$ as vectors: vector $f =
[f_1,f_2,\dots ,f_k]$ corresponds to the function 
$f =\sum _{i=1}^kf_i\chi_{J_i}.$ Let $Q^{(k)}$ be the 
isometric projection of $L^1$ onto $L^{(k)}$:

\begin{equation}\label{Q}
   Q^{(k)} (f)=\sum_{i=1}^k\left(\frac{1}{\lambda(J_i)}
\int_{J_i}f d\lambda \right)\chi_{J_i}=\left[\frac{1}{\lambda(J_1)}\int_{J_1}f d\lambda,\dots, 
\frac{1}{\lambda(J_k)}\int_{J_k}f d\lambda \right] 
\end{equation}

 Let $f = [f_1, f_2, \dots ,f_k] \in L^{(k)}.$ We define the operator $P_{\tau_n}^{(k)}: L^{(k)}\to L^{(k)}$ by 
\begin{eqnarray}\label{Approx_FP}P_{\tau_n}^{(k)} f
&=&  \left[f_1, f_2, \dots, f_k\right]\cdot \left(\mathbb{M}^{(k)}_{n}\right)^{\text{T}},
\end{eqnarray}
which is a finite dimensional approximation to the operator $P_{\tau_n}.$  $A^{\text{T}}$ denotes the transpose of the matrix $A.$


The following Lemma will be used several times in the sequel.

\begin{lemma}\label{convL1}
Let $\{g_n\}_{n=1,2\dots}$ be a sequence of non-increasing functions uniformly bounded in $L^\infty$. 
If $g_n\to h$, as $n\to\infty$, weakly in $L^1$, then  $g_n\to h$, as $n\to\infty$,  in $L^1$ and  almost everywhere (a. e.).
\end{lemma}
\begin{proof}
Since $g_n$'s are non-increasing and uniformly bounded in $L^\infty$, they are also of uniformly bounded variation.
By Helly's Theorem \cite{Rudin} there is a   subsequence $g_{n_k}$ convergent a.e. to some function $h_1$. Since $g_{n_k}\to h_1$ weakly in $L^1$ we have $h_1=h$.
Considering all possible subsequences we prove that $g_n\to h$ a.e. Since $g_n$ converge to $h$ a.e. and they are uniformly bounded in $L^\infty$, the convergence is also in $L^1$
(e.g. by Lebesgue Dominated Convergence theorem).
\end{proof}

Since each map $\tau_n$ is exact \cite{LY1}, by Proposition 1.2 of \cite{HM} we obtain that the invariant densities $f_{n,k}$ of $P_{\tau_n}^{(k)}$ are unique.

\begin{lemma}\label{decreasing}
The invariant density  $f_{n,k}$ of $P_{\tau_n}^{(k)}$ is non-increasing for any $n,k >1$.
\end{lemma}
\begin{proof}
Let $f\in L^{(k)}$. Since $ P_{\tau_{n}}^{(k)}f=Q^{(k)}P_{\tau_{n}}  f$, and both operators  $P_{\tau_{n}} $ and $Q^{(k)}$ transform non-increasing functions into non-increasing functions,
the operators $P_{\tau_{n}}^{(k)}$ also have this property. Let $f=1$ be a constant function understood as $[1,1,\dots,1]\in L^{(k)}$. It is non-increasing. Thus, all the functions
$(P_{\tau_{n}}^{(k)})^m f$, $m=1,2,\dots$, are non-increasing. Similarly as the estimate (\ref{unif est0}) was obtained, they can be shown to be uniformly bounded in $L^\infty$ and thus weakly compact in $L^1$. Then, Yosida-Kakutani theorem \cite{YK} shows that the sequence $1/s \sum_{m=1}^s (P_{\tau_{n}}^{(k)})^m f$ converges in $L^1$ to the invariant density  $f_{n,k}$.
By Lemma \ref{convL1} the convergence is also a.e. and  $f_{n,k}$ is non-increasing.
\end{proof}

 Using Ulam's method and corresponding convergence analysis described in \cite{TYL, M, D}, the following theorem can be proved.

\begin{theorem}\label{Thm4} 
Let $\tau \in \mathcal T_{pc}^{\infty,0}(I)$ be a piecewise convex map with countably many branches.
Let $\{\tau_n\}_{n=1}^\infty$ be the approximating sequence of piecewise
convex maps with a finite number of branches where $\tau_n$  are defined  in Equation (\ref{seqapprox}) in   Subsection \ref{Sec4}. If $f_{n,k}$ is a normalized fixed point of $P_{\tau_n}^{(k)}, k=1, 2, \dots,$ defined in  (\ref{Approx_FP}), then the sequence $\{f_{n,k}\}_{k=1}^\infty$ is weakly pre-compact in $L^1.$ Any limit point $f_n^*$ of the sequence $\{f_{n, k}\}_{k=1}^\infty$  is a fixed point of $P_{\tau_n}.$
\end{theorem}

\begin{proof}
 Let $P_{\tau_{n}}^{(k)}$ be the Ulam's approximation of the F-P operator $P_{\tau_{n}}$ of $\tau_{n}.$ Let $Q^{(k)}$ be the isometric projection defined in (\ref{Q}). It can be shown that (see: (4), page 3 \cite{M}; def. (2.1), page 5 \cite{TYL}): 
 \begin{equation}\label{composition}
   P_{\tau_{n}}^{(k)}Q^{(k)}f=Q^{(k)}P_{\tau_{n}} Q^{(k)} f. 
 \end{equation}
 By the definition of $Q^{(k)}$ it is obvious that
 \begin{equation}\label{Qnbound}
 \parallel Q^{(k)} f \parallel_{\infty} \le \parallel f \parallel_{\infty}.
 \end{equation}
By equations (\ref{LYI}), (\ref{composition}) and (\ref{Qnbound}) $P_{\tau_{n}^{(k)}}$ satisfies the following Lasota-Yorke type inequality
 \begin{equation}
     \parallel P_{\tau_{n}}^{(k)}f\parallel_{\infty} \le \left(a_{n}+D_{1}\right) \parallel f \parallel_{\infty} + D \parallel f \parallel_{1},
 \end{equation}
which implies the following uniform estimate
\begin{equation}\label{unif est}
 \parallel f_{n,k} \parallel_{\infty} \le \frac{D}{1-\left(a_{n}+D_{1}\right)}.
\end{equation}
 Now, the proof continues exactly as in \cite{M}.
\end{proof}

\begin{theorem}\label{Thm5} 
Let $\tau \in \mathcal T_{pc}^{\infty,0}(I)$ be a piecewise convex map with countably many branches. As described at the beginning of Subsection 3.1, let $\{\tau_n\}_{n=1}^\infty$ be the approximating sequence of piecewise
convex maps with finite numbers of branches. Let $P_{\tau_n}^{(k)}$ , $k=1,2,\dots$  be the sequence of Ulam's operators approximating the operators $P_{\tau_n}$. Let $f_{n,k}$ be the normalized (in $L^1$) fixed point of 
$P_{\tau_n}^{(k)}$. Then, the family  $\{f_{n,k}\}_{n=1,2,\dots, k=1,2,\dots}$ is weakly compact in $L^1$
and uniformly bounded in $L^\infty$. If $f_{n_j,k_j}$, $j=1,2,\dots $ is a weakly convergent subsequence, then it converges in $L^1$ (and almost everywhere) to a function $f$ which is a fixed point of $P_\tau$, $P_\tau f=f$.
\end{theorem}

\begin{proof}
The stationary densities $\{f_n\}_{n\ge 1}$ of   $\{\tau_n\}_{n\ge 1}$ are uniformly bounded in $L^\infty$ by Theorem \ref{Thm3}. 
 Moreover, the densities  $f_{n,k}$, the piecewise constant approximations of $f_n$'s  are also uniformly bounded in $L^\infty$  by formula (\ref{unif est}). Thus, the set $\{f_{n,k}\}_{k=1}^\infty$ is weakly compact in $L^1$. Assume that $\{f_{n,k}\}_{k=1}^\infty$  has weakly convergent subsequence $\{f_{n_j, k_j}\}$ with limit $f.$ Since the functions $\{f_{n_j, k_j}\}$ are  decreasing and uniformly bounded, by 
Lemma \ref{convL1} they converge to $f$ a.e. and in $L^1$.

It remains to show that $f$  is a fixed point of $P_\tau$, $P_\tau f=f.$  We will show that the measures $f d\lambda$ and $(P_{\tau}f)d\lambda$ are equal. It is enough to show that for any $g\in C(I)$, we have $\int g(f-P_{\tau}f)d\lambda=0.$ To simplify the notation we assume that the whole sequence $f_{n,k} $ converges to $f$.  We have,
\begin{eqnarray*}
&&\left|{\int g(f-P_{\tau}f)d\lambda}\right|\le\left|{\int g(f-f_{n, k})d\lambda}\right| +\left|{\int g(f_{n, k}-P_{\tau_n}^{(k)}f_{n,k})d\lambda}\right| \nonumber\\ &&+\left|{\int g(P_{\tau_n}^{(k)}f_{n,k}-P_{\tau_n}f_{n,k})d\lambda}\right| +\left|{\int g(P_{\tau_n}f_{n,k}-P_{\tau_n} f)d\lambda}\right|
+\left|{\int g(P_{\tau_n}f-P_\tau f)d\lambda}\right|.
\end{eqnarray*}
Since $f_{n,k}\to f$ in $L^1, $ the first and the fourth term go to $0$ as $n,k\to\infty.$  Since $f_{n,k}$ are the fixed points of $P_{\tau_n}^{(k)},$ the second term is $0.$  We have
\begin{eqnarray*}
{\int g\left[P_{\tau_n}^{(k)}f_{n,k}-P_{\tau_n}f_{n,k}\right]d\lambda}={\int g\left[Q^{(k)}(P_{\tau_n}f_{n,k})-P_{\tau_n}f_{n,k}\right]d\lambda}
=\int\left [Q^{(k)} g - g\right] P_{\tau_n}f_{n,k} d\lambda,
\end{eqnarray*}
by the self adjointness of $Q^{(k)}$.
Since $Q^{(k)}g \to g$ in $L^1$ and the densities $\{f_{n,k}\}$ are uniformly bounded,
$\int g\left[P_{\tau_n}^{(k)}f_{n,k}-P_{\tau_n}f_{n,k}\right]d\lambda$ converges to 0  as $k\to +\infty$.
By the property (vi) of the F-P operator we have
\begin{eqnarray*}
\left|{\int g(P_{\tau_n}f-P_\tau f)d\lambda}\right|=\left|{\int (g\circ \tau_n  -g\circ \tau) f d\lambda}\right|=\left|{\int_{[0,a_n]} (g\circ \tau_n  -g\circ \tau) f d\lambda}\right|
\le a_n \cdot 2M \cdot \|f\|_1 ,
\end{eqnarray*}
where $M= \sup |g|$. 
This shows that$\left|{\int g(P_{\tau_n}f-P_\tau f)d\lambda}\right|$  converges to 0 as $n\to +\infty$.
\end{proof}

\section{Examples}\label{Sec6} 

\begin{example}\label{Ex1}
\noindent Consider the  piecewise expanding and piecewise linear   map $T:[0,1]\rightarrow [0,1]$  with countable number of branches defined as 
\begin{equation}
    T(x)=i(i+1)\left(x-\frac{1}{i+1}\right) \hspace{3 mm}\  \text{ on}\  \hspace{2 mm}\left [\frac{1}{i+1},\frac{1}{i}\right], i=1,2, \cdots.
\end{equation}
It is easy to show that the Lebesgue measure is invariant under $T$  and 
$f = \bold{1}$ is the invariant density of $T$. Now, consider the conjugation $h : [0, 1] \to [0, 1]$ defined by $h(x) = 1-(1-x)^2.$ We construct the
piecewise convex map $\tau : [0, 1] \to [0, 1] $ with countable number of branches defined by  $\tau = h^{-1}\circ T\circ h.$ See Figure 1 for a graph of $\tau.$
\begin{figure}[ht!]
    \centering
    \includegraphics[width=9cm, height =8cm]{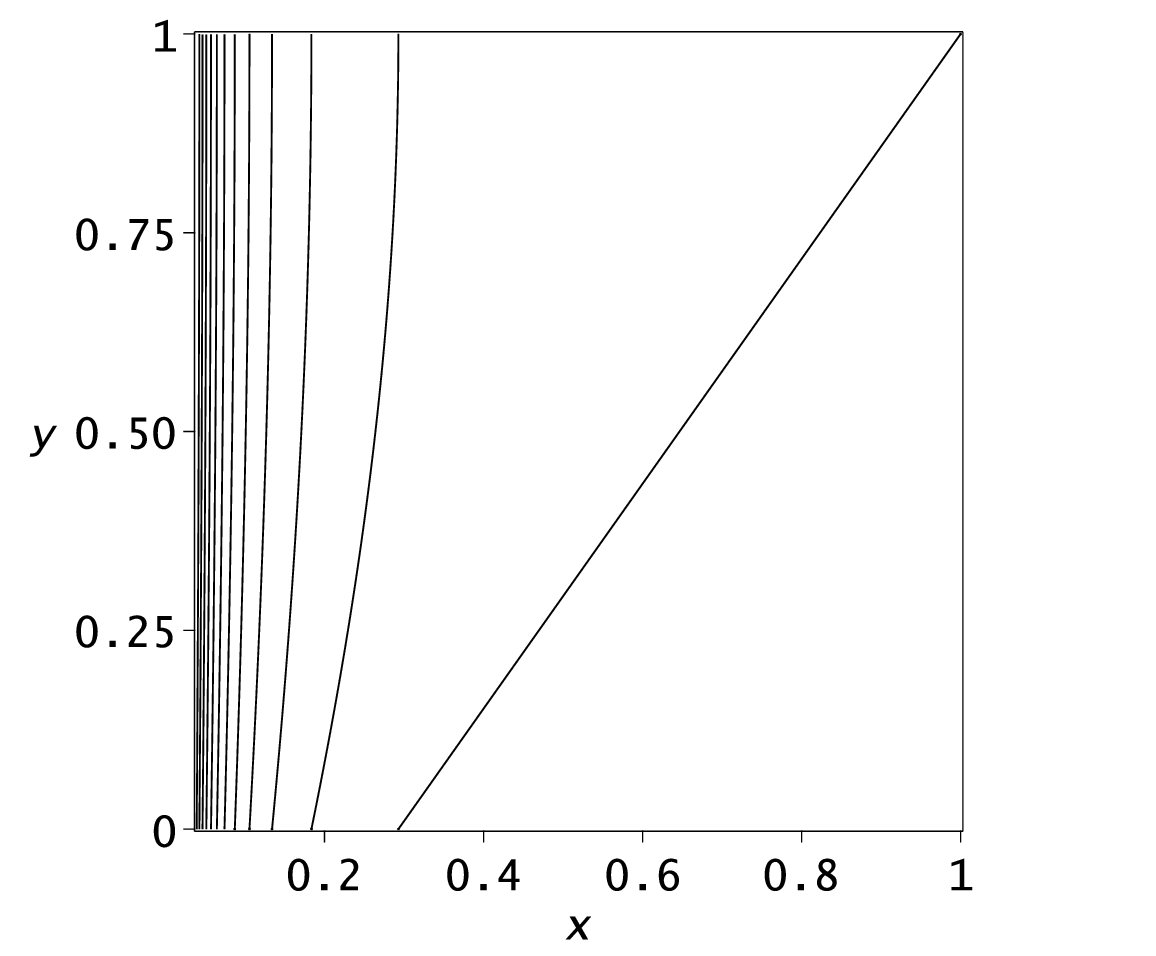}
    \caption{The graph of the piecewise convex map  $\tau$ with countably infinite number of branches for Example \ref{Ex1}.}
\end{figure}
The piecewise convex map $\tau$ is topologically conjugate to the piecewise linear and piecewise expanding map $T$ via the conjugation $h$. Therefore, the stationary  density $g$ of $\tau$  is given by
$g = f \circ  h \times \abs{h^\prime}.$ Now, $h(x) = 1-(1-x)^2.$ Hence, $g(x) = f(h(x)) \times \abs{h^\prime(x)}=\abs{2(1-x)}.$

The map $\tau$ is defined on the partition 
$$\left \{\left [1-\sqrt{\frac{i}{i+1}},1-\sqrt{\frac{i-1}{i}}\right]\right\}_{i=1}^\infty,$$
i.e., the partition points are $$\left\{\dots, 1-\sqrt{\frac{4}{5}}, 1-\sqrt{\frac{3}{4}}, 1-\sqrt{\frac{2}{3}},1-\sqrt{\frac{1}{2}}, 1\right\}.$$

We have $$\tau(x)=1-\sqrt{1-i^2+i(i+1)(1-x)^2}  \  \text{ for }\ x\in \left [1-\sqrt{\frac{i}{i+1}},1-\sqrt{\frac{i-1}{i}}\right]\ ,\ i=1,2,\dots $$
The first few, starting from 1 to the left,   branches of $\tau$ 
are 
\begin{eqnarray*}  \tau_1(x)&=1-\sqrt{2}(1-x) \  & \text{for}\  1-\sqrt{1/2}\le x\le 1 ; \\
 \tau_2(x)&= 1-\sqrt{6(1-x)^2-3} \  & \text{for}\  1-\sqrt{2/3}\le x< 1-\sqrt{1/2} ;\\
 \tau_3(x)&=1-\sqrt{12(1-x)^2-8} \   &\text{for}\ 1-\sqrt{3/4} \le x<1-\sqrt{2/3} ;\\
 \tau_4(x)&=1-\sqrt{20(1-x)^2-15} \  & \text{for}\ 1-\sqrt{4/5} \le x<1-\sqrt{3/4} .\\
\end{eqnarray*}
Now, consider the following   sequence $\{\tau_n\}_{n\ge 0}$   of piecewise convex maps $\tau_n:[0,1]\to [0,1]$  with a finite number of branches:

$$\tau_n(x)=\begin {cases} \frac{1}{1-\sqrt{\frac{n}{n+1}}}x  \ ,         \  0\le x< 1-\sqrt{\frac{n}{n+1}};\\
                           \tau(x) \ ,         \ 0\le x< 1-\sqrt{\frac{n}{n+1}}\le x\le 1.
\end{cases}
$$
See Figure 2 for a graph of $\tau_n$ with $n=5.$  
\begin{figure}[ht!]
    \centering
    \includegraphics[width=9cm, height =8 cm]{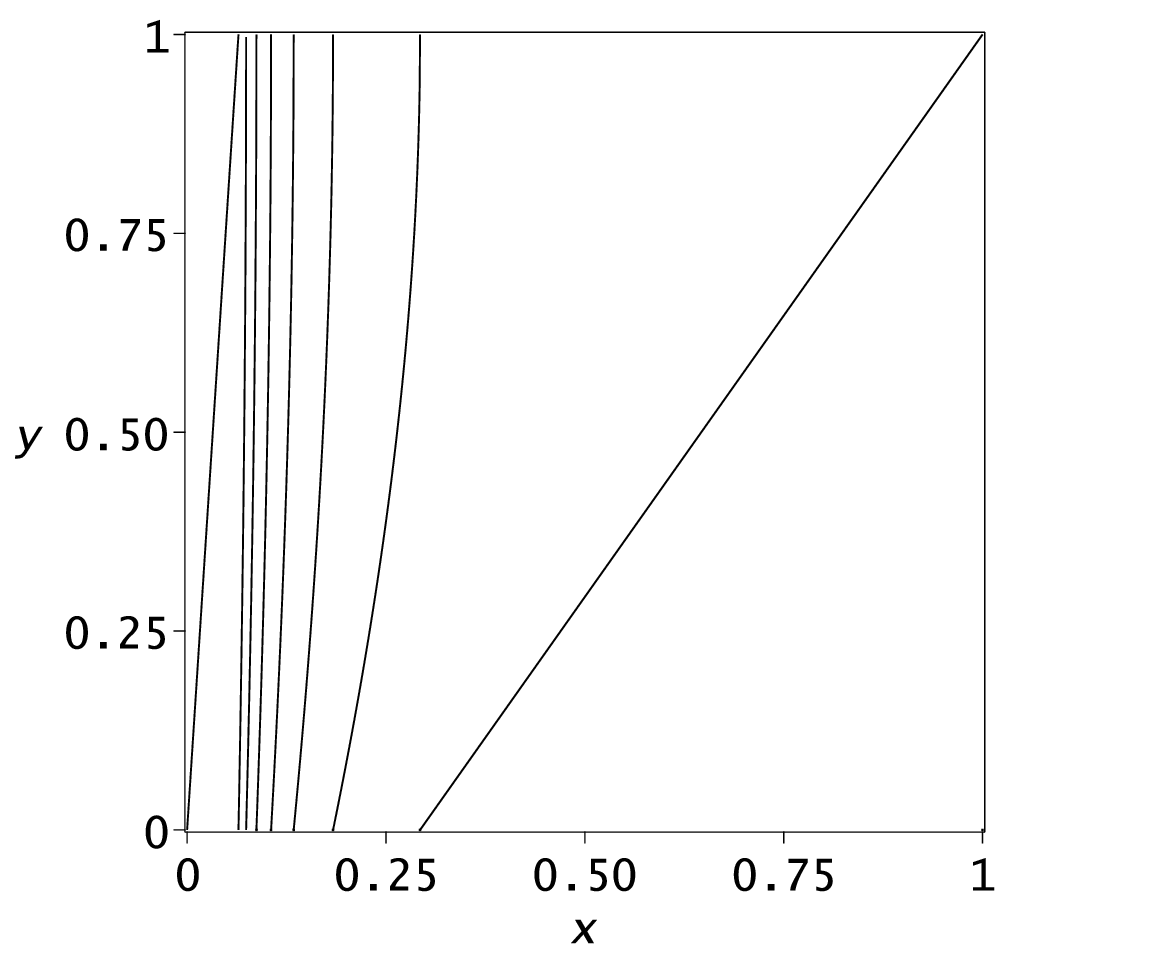}
    \caption{ The graph of the piecewise convex map  $\tau_n$ with a finite  number of branches ($n=8$).}
\end{figure}
The  sequence $\{\tau_n\}_{n\ge 0}$   of piecewise convex map $\tau_n:[0,1]\to [0,1]$  with finite number of branches converges almost uniformly to $\tau$ with the countable number of branches.

In Figure 3, we present  the graph
of the actual density $f^*$ (in red) of the piecewise convex map $\tau$ with countable number of branches and  graphs
of the approximate stationary densities  $ f_{n,k}$ (in blue)  via Ulam's method for  maps $\tau_n$  with  a finite number of branches.  Numerical computations are performed for a number of  cases. 
In the following table, we present  the $L^1$ norm error
$\parallel f^* - f_{n,k}\parallel_1.$ Note that for each $n$, 
$\tau_n$  is a map with a finite number of branches which approximates  the piecewise convex map $\tau$  in Figure 1 with  countable number of branches.
\begin{figure}[ht!]
    \centering
    \includegraphics[width=7cm, height = 8cm]{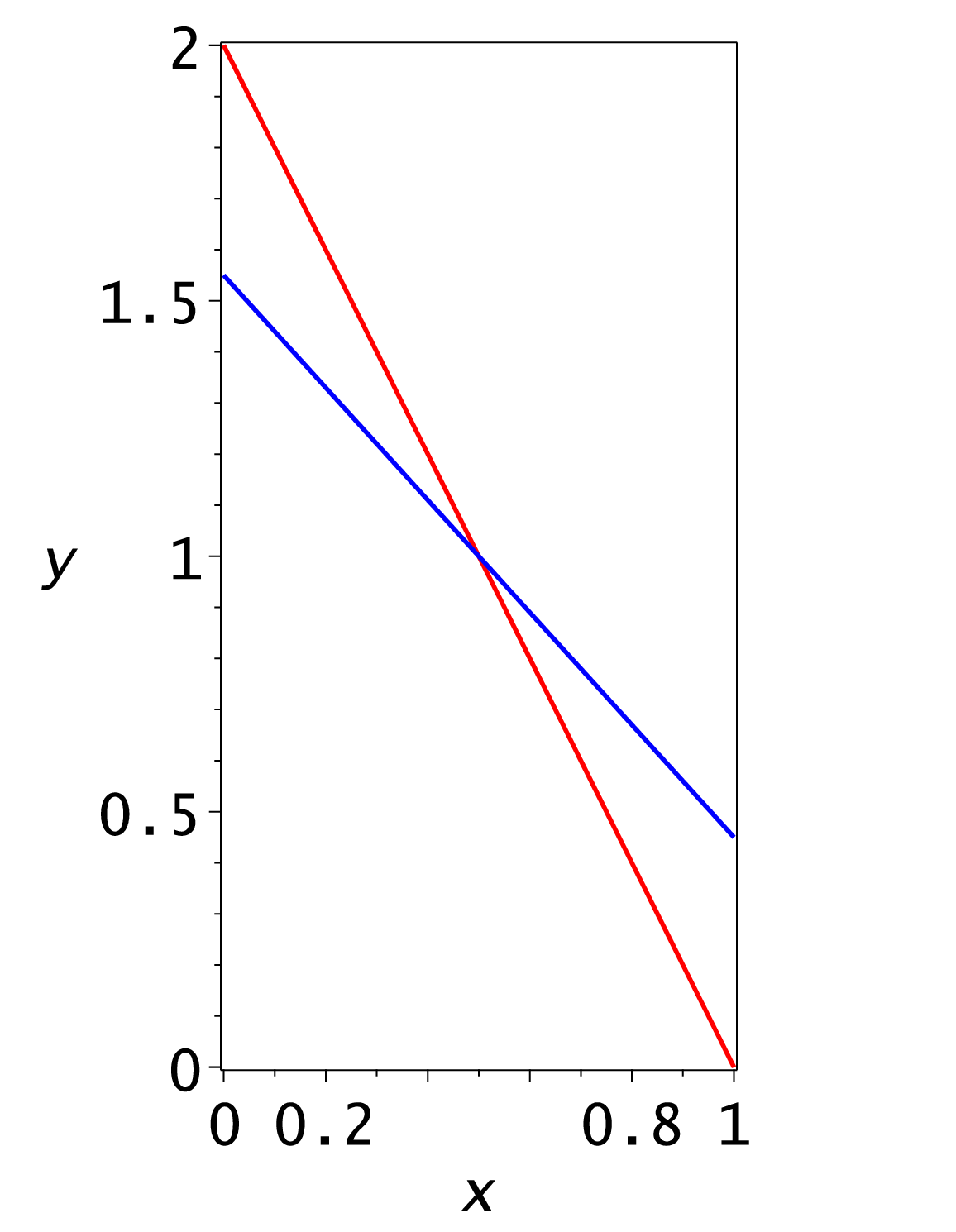}
       \includegraphics[width=7cm, height = 8cm]{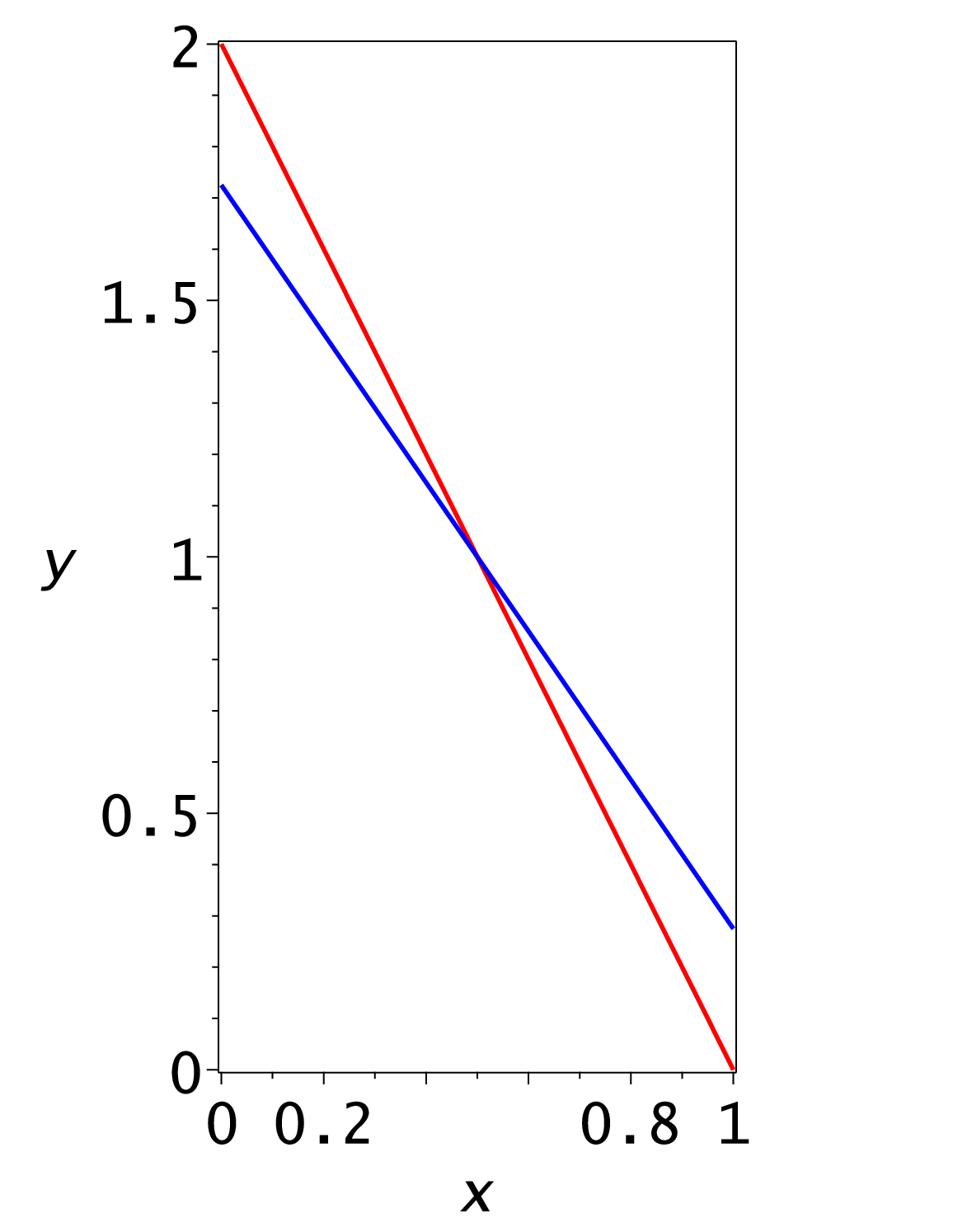}
    \caption{ The graph of the actual invariant density $f^*$ of the piecewise convex map $\tau$  with infinite number of branches (in red) and the graph of the 
approximating density $f_{n,k}$ (in blue): $ n = 5, k = 1000$ on the left and $ n = 10, k = 1000$ on the right hand side.}
\end{figure}

\begin{center}
\begin{tabular}{|c|c|c|}\hline
 $n$& $k$&$\parallel f^* - f_{n,k}\parallel_1$ \\\hline
  $5$ & 1$00$&$0.2195470623$\\\hline
  $5$ & $1000$&  $0.2195243505$\\\hline
$6$& 100&$0.1943541673$\\\hline
 $6$& 1000&$0.1943541673$\\\hline
  $7$& 1000&$0.1742407352$\\\hline
  $10$& 1000&$0.133493040$\\\hline
$12$& 1000&$0.11551923$\\\hline
 \end{tabular}.
\end{center}
The above table shows that as we increase $n$,  the $L^1$ norm error
$\parallel f^* - f_{n,k}\parallel_1$ gets smaller. For the fixed $n$, the increasing of $k$ is not effective. For example, the errors for $k=100$
and $k=1000$  are almost the same. The main method to lower the error is to increase the number of branches of $\tau_n$. 

By Theorem \ref{Thm5}, $L^1$ norm error can be made arbitrarily small by making $n$ and $k$ large enough.  

\end{example}

\begin{example}\label{Ex2}
    
Consider the  piecewise convex map $\tau:[0,1]\rightarrow [0,1]$  with countable number of branches defined as 
\begin{equation}
    \tau(x)=\frac{1}{\frac{2i+1}{i(i+1)}-x}-i \hspace{3 mm} \text{on}  \hspace{2 mm} \left[\frac{1}{i+1},\frac{1}{i}\right],\ i=1,2, \cdots.
\end{equation}
See Figure 4. The map  $\tau$ is defined on  the countable  partition  of $[0, 1]$,  with partition points $a_0=1, a_1=\frac 12, a_2=\frac 13, \cdots, a_n=\frac {1}{n+1}, \cdots.$
\begin{figure}[ht!]
    \centering
    \includegraphics[width=8cm, height = 8cm]{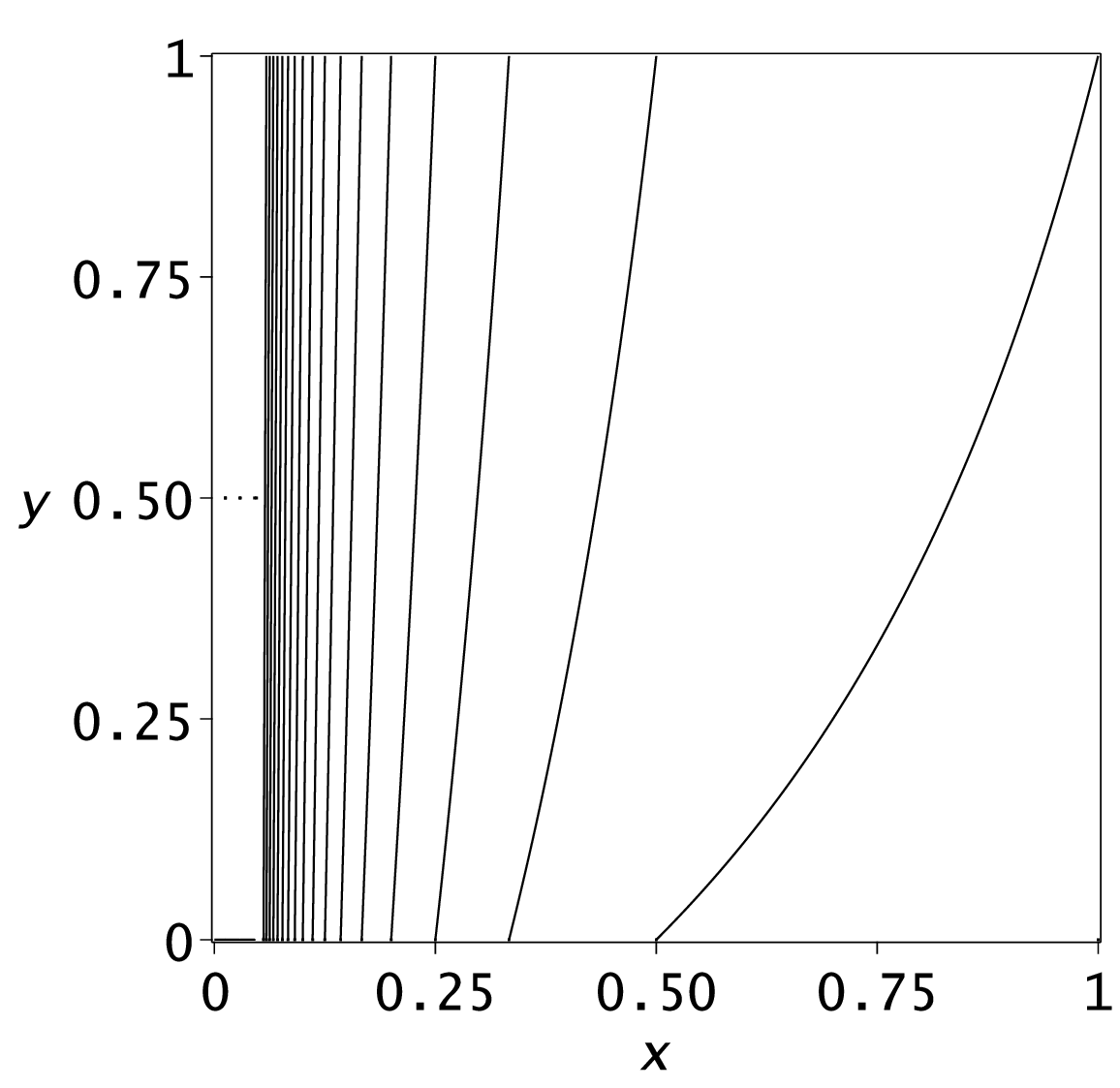}
    \caption{The graph of the piecewise convex map with countable number of branches for Example \ref{Ex2}.}
\end{figure}

It is shown in \cite{GIR} that   $\tau\in \mathcal{T}_{pc}^{\infty,0}(I)$ and hence by Theorem \ref{ThmSec3}, $\tau$ has an ACIM. The actual density of $\tau$ is not known. The  sequence $\{\tau_n\}_{n\ge 0}$   of piecewise convex maps $\tau_n:[0,1]\to [0,1]$  with finite number of branches, see Figure 5   for  graphs of $\tau_n$ with $n=8$ and $n=10$,  converges almost uniformly to $\tau,$ where 

$$\tau_n(x)=\begin {cases} (n+1)x  \ ,         \  0\le x< \frac{1}{n+1};\\
                           \tau(x) \ ,         \  \frac{1}{n+1}\le x\le 1.
\end{cases}
$$

\begin{figure}[ht!]
    \centering
    \includegraphics[width=8cm, height =8cm]{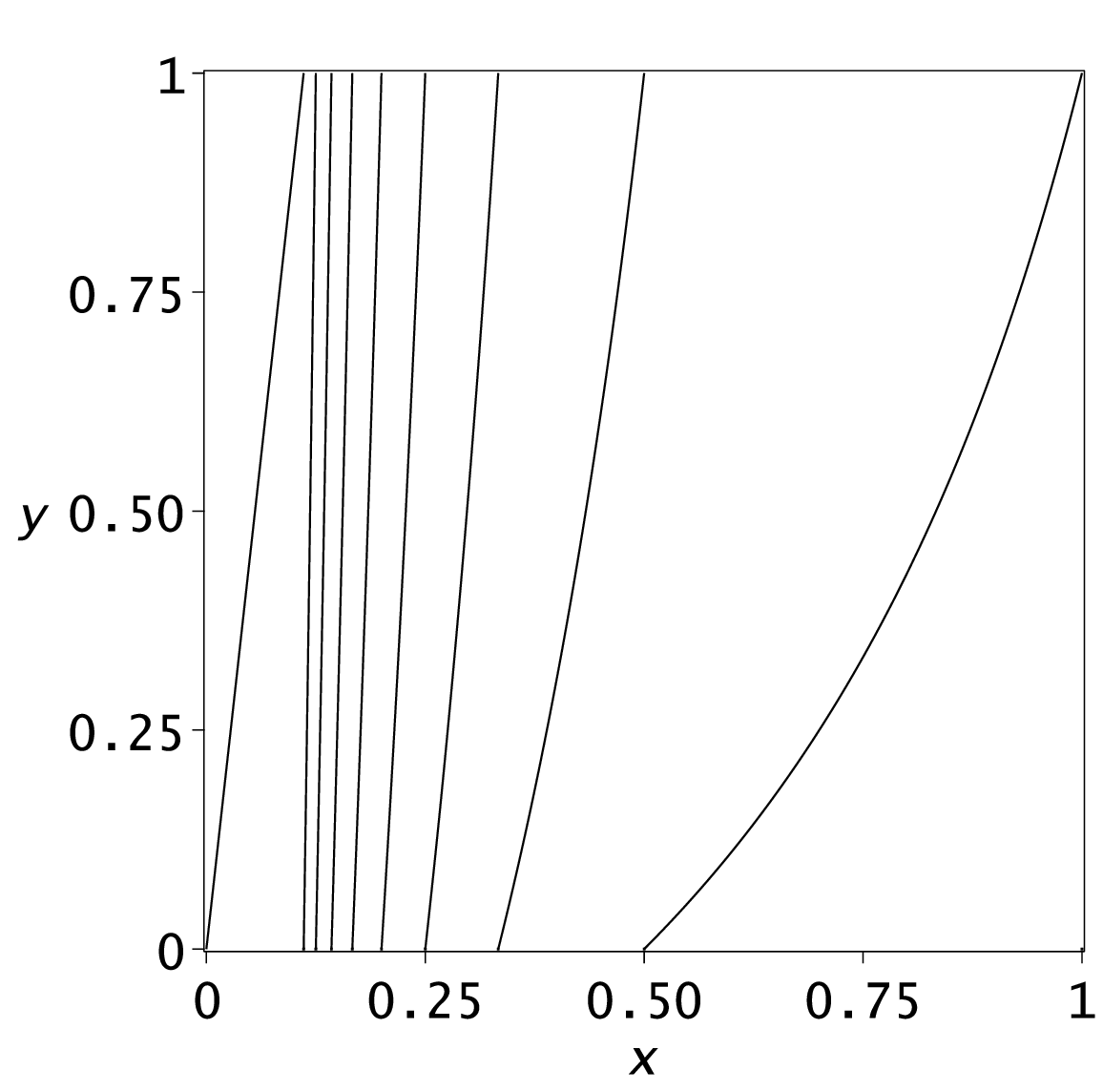}
   \includegraphics[width=8cm, height = 8cm]{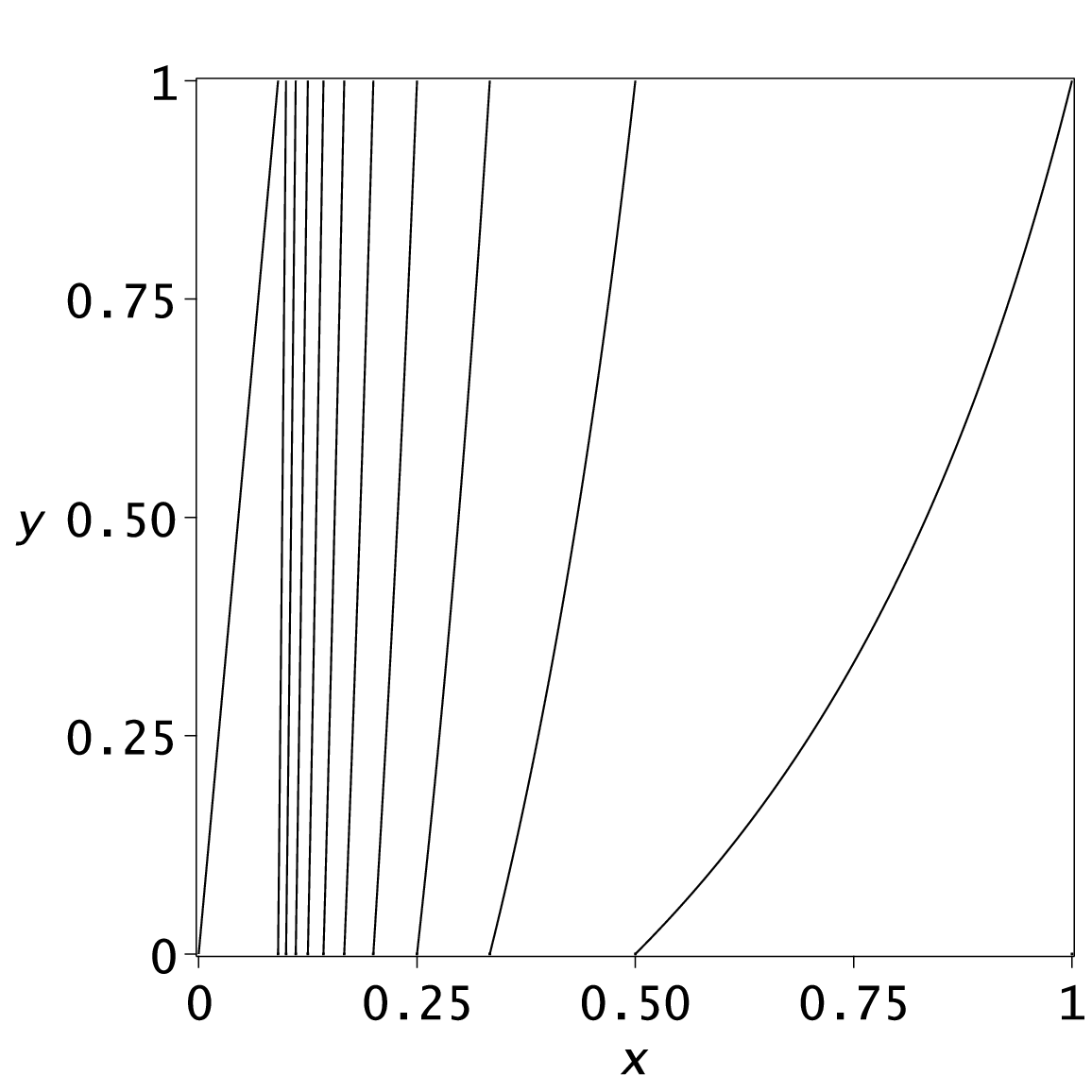}
\caption{The graphs of the piecewise convex maps  $\tau_n$ with finite  number of branches $n=8$ and $n=10$.}
\end{figure}

In Figure 6, we present the graph of the Ulam's approximation  $f_{n,k}, n=10, k=1000$  of the actual  invariant density of $f_n, n=10$, of the piecewise convex map $\tau_n, n=10$, with finite number of branches. By Theorem \ref{Thm5} it is also an approximation of the invariant
density $f^*$ of the map $\tau$.  
The same Figure 6 shows also the approximation $f_{n,k}, n=10, k=500$, as the densities $f_{10,1000}$ and $f_{10,500}$ are indistinguishable at this scale. We have $\|f_{10,1000}-f_{10,500}\|_{1}\sim 0.00055$. In Figure 7 we show the enlargement of both graphs on a small subinterval. The scales on $x$, $y$ axes are different.

\begin{figure}[ht!]
    \centering
    \includegraphics[width=7cm, height = 9cm]{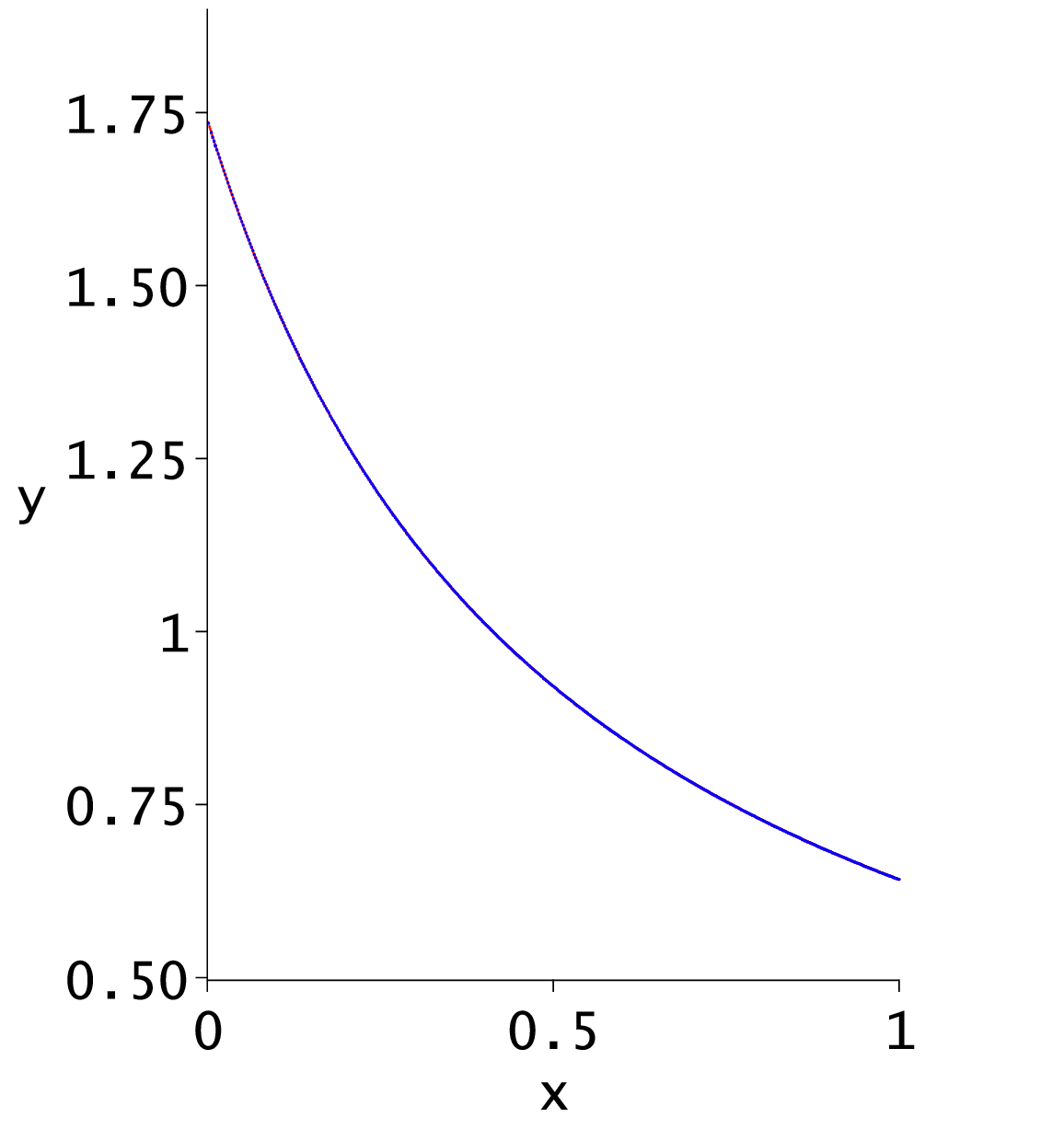}
    \caption{The graph of the Ulam's approximation  $f_{n,k}, n=10, k=1000$.   }
\end{figure}

Similarly as in Example \ref{Ex1}, the effect of increasing $k$ (after some threshold) is negligible. The improvement of the approximation is  achieved by increasing the parameter $n$. The density $f_{11,1000}$ is indistinguishable from $f_{10,1000}$ at the paper illustrations scale
with  $\|f_{10,1000}-f_{11,1000}\|_{1}\sim 0.00035$. The error $\|f_{11,1000}-f_{12,1000}\|_{1}\sim 0.00029.$ 

\begin{figure}[ht!]
    \centering
    \includegraphics[width=9cm, height = 9cm]{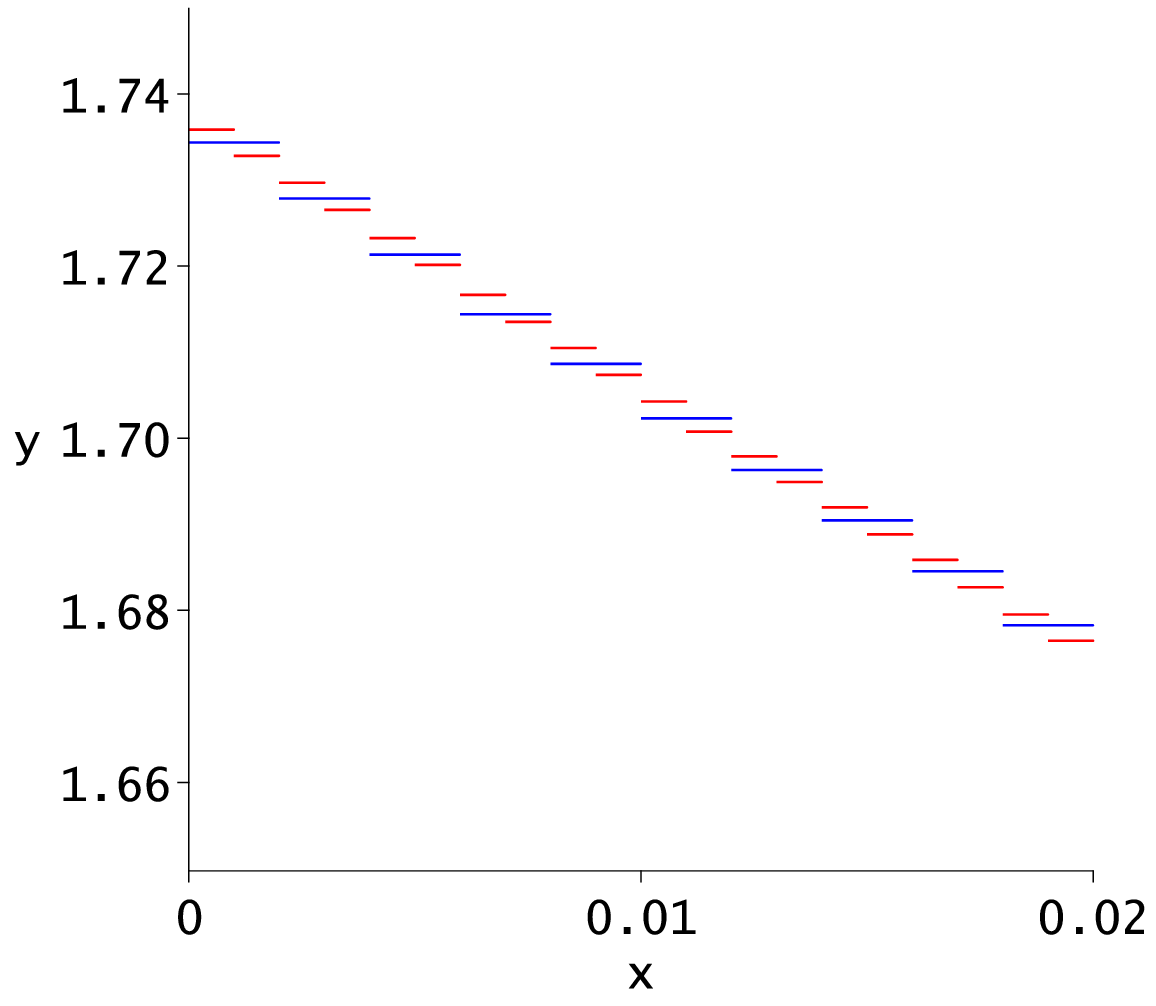}
    \caption{Enlargement of the approximating densities  $f_{10,1000}$ (in red) and $f_{10,500}$ (in blue) on $[0,0.02].$}
\end{figure}
\end{example}

\noindent{\bf Acknowledgment}: The authors are very grateful to  anonymous reviewers for very detailed and useful comments and suggestions and also for  stating a new problem related with their note. The research was supported by the NSERC Discovery Grant (DG) of the first author and second  author. The third author is grateful to Concordia University and University of Prince Edward Island for financial support via NSERC DG grants of the first author and the second author respectively.

\newpage

\newpage

\end{document}